\documentclass[11pt]{amsart}

\usepackage{amsmath, amsfonts, amssymb,amsthm}
\usepackage{amstext}
\usepackage{mathrsfs}

\usepackage{float,epsf,subfigure}
\usepackage[all,cmtip]{xy}
\usepackage{hyperref}
\usepackage{algorithm}
\usepackage{algorithmic}
\usepackage{mathtools}
\usepackage{float}
\usepackage{bm}
\theoremstyle{plain}
\newtheorem{theorem}{Theorem}[section]

\newtheorem{cor}[theorem]{Corollary}
\newtheorem{def-thm}[theorem]{Definition-Theorem}
\newtheorem{lemma}[theorem]{Lemma}

\newtheorem*{tha}{Theorem A}

\theoremstyle{definition}

\newtheorem{remark}[theorem]{Remark}

\def\min{\mathop{\mathrm{min}}}

\begin{document}
\title[Carlson-Griffiths'  theory  via Brownian motion]{Carlson-Griffiths'  theory via Brownian motion}
\author[X.J. Dong]
{Xianjing Dong}

\address{Academy of Mathematics and Systems Sciences \\ Chinese Academy of Sciences \\ Beijing, 100190, P.R. China}
\email{xjdong@amss.ac.cn}


\date{}
\maketitle \thispagestyle{empty} \setcounter{page}{1}

\begin{abstract}
\noindent  Early in 1970s, Carlson-Griffiths  made a significant progress in the study
of Nevanlinna theory, who devised equi-distribution theory for  holomorphic
mappings from $\mathbb C^m$ into a projective algebraic manifold intersecting
divisors. In this paper, we   develop Carlson-Griffiths' theory by generalizing the source manifold $\mathbb C^m$ to complete K\"ahler manifolds.
\end{abstract}

\vskip\baselineskip

\setlength\arraycolsep{2pt}

\section{Introduction}
 Nevanlinna theory, devised  by R. Nevanlinna in 1925, is part of the theory of meromorphic functions which generalizes the
 Picard's little theorem. This theory was later generalized to
   parabolic manifolds by
 Stoll \cite{Stolla, S21}.
Early in 1970s, Carlson and Griffiths \cite{gri, gth} made a significant progress in the study of  Nevanlinna theory, who devised the equi-distribution theory of  holomorphic mappings from
$\mathbb C^m$ into complex projective algebraic manifolds intersecting divisors.
Later,  Griffiths and King \cite{gri1,gth} proceeded to generalize the theory to  affine algebraic manifolds.
More generalizations
were done by Sakai \cite{Sakai} in terms of Kodaira dimension,
the singular divisor was considered  by Shiffman \cite{Shiff}.
Now let's first  review Carlson-Griffiths'  work briefly.

Let $V$ be a complex projective algebraic manifold satisfying $\dim_{\mathbb C}V\leq m$. In general, we set for two  holomorphic line bundles $L_1, L_2$ over $V$ that 
\begin{eqnarray*}
\overline{\left[\frac{c_1(L_2)}{c_1(L_1)}\right]}&=&\inf\left\{s\in\mathbb R: \ \omega_2<s\omega_1;  \ ^\exists\omega_1\in c_1(L_1),\  ^\exists\omega_2\in c_1(L_2) \right\}, \\
\underline{\left[\frac{c_1(L_2)}{c_1(L_1)}\right]}&=&\sup\left\{s\in\mathbb R: \ \omega_2>s\omega_1;  \ ^\exists\omega_1\in c_1(L_1),\  ^\exists\omega_2\in c_1(L_2)\right\}.
\end{eqnarray*}

 Let $f:\mathbb C^m\rightarrow V$ be a holomorphic mapping. We use  $\delta_f(D)$ to denote the defect
 of $f$ with respect to $D,$ defined by
$$\delta_f(D)=1-\limsup_{r\rightarrow\infty}\frac{N_f(r,D)}{T_f(r,L)},$$
where $N_f(r,D),$ $T_f(r,L)$ are defined in Remark \ref{remark}.
Carlson-Griffiths proved
\begin{tha}\label{} Let $f:\mathbb C^m\rightarrow V$ be a differentiably non-degenerate holomorphic mapping.
  Let $L\rightarrow V$ be a positive line bundle and let a divisor $D\in|L|$  be of  simple normal crossing type. Then
  $$\delta_f(D)\leq \overline{\left[\frac{c_1(K_V^*)}{c_1(L)}\right]}.$$
\end{tha}
The purpose of this paper is to generalize  Theorem A to complete  K\"ahler manifolds.
  Our method
is to combine  Logarithmic  Derivative Lemma (LDL) with stochastic technique developed by  Carne and Atsuji.
So, the first task here is  to establish LDL for meromorphic functions on
 complete  K\"ahler manifolds (see Theorem \ref{log1} below), which may be of its own interest.
 Recall that the first probabilistic proof of Nevanlinna's Second Main Theorem of meromorphic functions on $\mathbb C$ is due to  Carne \cite{carne}, who
 re-formulated  Nevanlinna's functions  in terms of Brownian motion. Later, Atsuji \cite{at, at1, at2, atsuji} obtained a Second Main Theorem of meromorphic functions defined on
 complete K\"ahler manifolds.
 Recently, Dong-He-Ru \cite{Dong} re-visited this technique and
    provided a  probabilistic proof of Cartan's Second Main Theorem.

We  introduce the main results in this paper, the notations will be provided in the later sections. Note by Remark \ref{remark} that the definitions of Nevanlinna's functions in the K\"ahler manifold case are natural extensions of the classical ones in the $\mathbb C^m$ case.
Indeed,  for  technical reasons, all the K\"ahler manifolds (as domains) considered in this paper are assumed to be open. 

Let $M$ be a complete K\"ahler manifold with $\dim_{\mathbb C}M\geq \dim_{\mathbb C}V.$
Let $L\rightarrow V$ be an arbitrary holomorphic line bundle and  fix a Hermitian metric $\omega$
 on $V.$
 
 We first establish the following Logarithmic Derivative Lemma:
\begin{theorem}\label{log1} Let
$\psi$ be a nonconstant meromorphic function on  $M.$   Then for any $\delta>0,$ there exists a function $C(o, r, \delta)>0$  independent of $\psi$ and a set  $E_\delta\subset(1,\infty)$ of finite Lebesgue measure such that
\begin{eqnarray*}
   m\Big(r,\frac{\|\nabla_M\psi\|}{|\psi|}\Big)&\leq& \Big{(}1+\frac{(1+\delta)^2}{2}\Big{)}\log T(r,\psi)+\log C(o,r,\delta)
\end{eqnarray*}
 holds for $r>1$ outside  $E_\delta,$ where $o$ is a fixed reference point in $M.$
\end{theorem}
The estimate of term $C(o, r, \delta)$ will be provided  when $M$ is non-positively curved (see (\ref{esti})).
    Let ${\rm{Ric}}_M$ and $\mathscr R_M$ be the Ricci curvature tensor and Ricci curvature form of $M$ respectively.
Set 
\begin{equation}\label{kappa}
  \kappa(t)=\frac{1}{2\dim_{\mathbb C}M-1}\min_{x\in \overline{B_o(t)}}R_M(x),
\end{equation}
where $R_M(x)$ is the
pointwise lower bound of Ricci curvature defined by
$$ R_M(x)=\inf_{\xi\in T_{x}M} \frac{{\rm{Ric}}_M(\xi,\bar{\xi})}{\|\xi\|^2}.$$

 We have the following Second Main Theorem (SMT) for complete K\"ahler manifolds:
\begin{theorem}\label{second}  Let a divisor $D\in|L|$  be of  simple normal
crossing type.  Let $f:M\rightarrow V$ be a differentiably non-degenerate meromorphic mapping.  Then for any $\delta>0,$ there exists a function $C(o, r, \delta)>0$  independent of $f$ and a set $E_\delta\subset(1,\infty)$ of finite Lebesgue measure such that
  \begin{eqnarray*}
      T_f(r,L)+T_f(r,K_V)+T(r,\mathscr{R}_M)
     &\leq& \overline{N}_f(r,D)+O\big{(}\log T_f(r, \omega)+\log C(o,r,\delta)\big{)}
  \end{eqnarray*}
  holds for $r>1$ outside  $E_\delta.$
  \end{theorem}
When $M$ is non-positively
curved, by estimating $C(o,r,\delta)$ and $T(r,\mathscr{R}_M),$  we obtain 
 \begin{theorem}\label{nonpositive} Let a divisor $D\in|L|$  be of  simple normal crossing type.
 Let $f:M\rightarrow V$ be a differentiably non-degenerate meromorphic mapping. Then for any $\delta>0$
  \begin{eqnarray*}
      T_f(r,L)+T_f(r,K_V)
     &\leq& \overline{N}_f(r,D)+O\big{(}\log T_f(r,\omega)-\kappa(r)r^2+\delta\log r\big{)}
  \end{eqnarray*}
  holds for $r>1$ outside a set  $E_\delta\subset(1,\infty)$ of finite Lebesgue measure.
\end{theorem}
We denote by $\Theta_f(D)$  another defect (without counting multiplicities) of $f$ with respect to $D,$ defined  by
$$\Theta_f(D)=1-\limsup_{r\rightarrow\infty}\frac{\overline{N}_f(r,D)}{T_f(r,L)}.$$
\begin{cor}[Defect relation]\label{defect} Assume the same conditions as in Theorem $\ref{nonpositive}$.
 If $f$ satisfies the growth condition
$$ \liminf_{r\rightarrow\infty}\frac{r^{2}\kappa(r)}{T_f(r,\omega)}=0,$$
then
$$\Theta_f(D)\underline{\left[\frac{c_1(L)}{\omega}\right]}\leq
\overline{\left[\frac{c_1(K^*_V)}{\omega}\right]}.$$
\end{cor}
In particular, if $M=\mathbb C^m,$ then we have $\kappa(r)\equiv0.$ So, Corollary \ref{defect} implies Theorem A.
More general, we have SMT for singular divisors:
\begin{theorem}\label{second1}
Let $D$ be a hypersurface of $V.$
Let  $f:M\rightarrow V$ be a differentiably non-degenerate meromorphic mapping.  Then
for any $\delta>0$
  \begin{eqnarray*}
   &&T_f(r,L_D)+T_f(r,K_V)-\overline{N}_f(r,D) \\
     &\leq& m_f\big{(}r,{\rm{Sing}}(D)\big{)}
   +O\big{(}\log T_f(r,\omega)-\kappa(r)r^2+\delta\log r\big{)}
  \end{eqnarray*}
  holds for $r>1$ outside a set  $E_\delta\subset(1,\infty)$ of finite Lebesgue measure.
\end{theorem}

\section{Preliminaries}
 We introduce some basics referred to
 \cite{bass,Bishop, mar,gri1,13,NN,itoo,polar}.
\subsection{Poincar\'e-Lelong formula}~

 Let $M$ be a $m$-dimensional complex manifold. A divisor $D$ on $M$ is said to be of \emph{normal crossings} if  $D$ is locally defined by an equation
$z_{1}\cdots z_{k}=0$ for a local holomorphic  coordinate system $z_1,\cdots,z_m.$ Additionally, if every irreducible component of $D$ is smooth,  one says that
$D$ is of \emph{simple normal crossings}.
A holomorphic line bundle $L\rightarrow M$ is said to be \emph{Hermitian} if $L$ is equipped with a Hermitian metric $h=(\{h_\alpha\},\{U_\alpha\}),$ where
$$h_\alpha: U_\alpha\rightarrow \mathbb R^+$$
 are positive smooth functions such that $h_\beta=|g_{\alpha\beta}|^2h_\alpha$ on $U_\alpha\cap U_\beta,$ and $\{g_{\alpha\beta}\}$ is a transition function system of $L.$ Let $\{e_\alpha\}$ be a local holomorphic   frame of $L,$  we have $\|e_\alpha\|^2_h=h_\alpha.$ A Hermitian metric $h$ of $L$ defines a global, closed and smooth  (1,1)-form $-dd^c\log h$ on $M,$ where
$$d=\partial+\bar{\partial}, \ \ d^c=\frac{\sqrt{-1}}{4\pi}(\bar{\partial}-\partial), \ \ dd^c=\frac{\sqrt{-1}}{2\pi}\partial\bar{\partial}.$$
We call $-dd^c\log h$ the Chern form denoted by $c_1(L,h)$ associated with metric $h,$ which determines a Chern class $c_1(L)\in H^2_{{\rm{DR}}}(M,\mathbb R)$, $c_1(L,h)$ is also called the curvature form of $L.$ If $c_1(L)>0,$ namely, there exists a Hermitian metric $h$ such that $-dd^c\log h>0,$ then we say that $L$ is positive, written as $L>0.$

  Let $T M$ denote the holomorphic tangent bundle of $M.$ The \emph{canonical line bundle} of $M$ is defined by $$K_M=\bigwedge^mT^*M$$
with transition functions
$g_{\alpha\beta}=\det(\partial z^\beta_j/\partial z^\alpha_i)$
on $U_\alpha\cap U_\beta.$
Given a Hermitian metric $h$ on $K_M$,  it well defines a global, positive and smooth $(m,m)$-form
$$\Omega=\frac{1}{h}\bigwedge_{j=1}^m\frac{\sqrt{-1}}{2\pi}dz_j\wedge d\bar{z}_j$$
on $M,$ which is therefore a volume form of $M.$ The Ricci form of $\Omega$ is defined by
${\rm{Ric}}\Omega=dd^c\log h.$
Clearly,
$c_1(K_M,h)=-{\rm{Ric}}\Omega.$
Conversely, if let $\Omega$ be a volume form on $M$ which is compact, there exists a unique Hermitian metric $h$ on $K_M$ such that $dd^c\log h={\rm{Ric}} \Omega.$

 Let $H^0(M,L)$ denote the vector space of  holomorphic global sections of $L$ over $M$.
For any $s\in H^0(M,L)$, the divisor $D_s$
is well defined  by $D_s\cap U_\alpha=(s)|_{U_\alpha}$.  Denote by $|L|$ the
\emph{complete linear system} of all effective divisors $D_s$ with $s\in  H^0(M,L).$ Let $D$ be a divisor on $M$, then
  $D$ defines a holomorphic line bundle $L_D$ over $M$ in such  manner: let $(\{g_\alpha\},\{U_\alpha\})$ be the local defining function system of $D,$ then the transition system is given by $\{g_{\alpha\beta}=g_\alpha/g_\beta\}.$ Note that $\{g_\alpha\}$ defines a global meromorphic  section on $M$ written as $s_D$ of $L_D$ over $M,$ called the \emph{canonical section} associated with   $D.$

We introduce the famous Poincar\'e-Lelong formula:
\begin{lemma}[Poincar\'e-Lelong formula, \cite{gri}]\label{} Let $L\rightarrow M$ be a holomorphic line bundle equipped 
with a Hermitian metric $h,$  and let $s$ be a holomorphic section of $L$ over $M$ with  zero divisor $D_s.$
Then $\log\|s\|_h$ is locally integrable on $M$ and it defines a current satisfying the current equation
  $$dd^c\log\|s\|_h^2=D_s-c_1(L,h).$$
\end{lemma}

\subsection{Brownian motions}~

 Let  $(M,g)$ be a Riemannian manifold with the Laplace-Beltrami operator $\Delta_M$ associated with  metric $g.$ 
A \emph{Brownian motion} $X_t$ in $M$
is a heat diffusion  process  generated by $\Delta_M/2$ with  \emph{transition density function} $p(t,x,y)$ being  the minimal positive fundamental solution of  heat equation
  $$\frac{\partial}{\partial t}u(t,x)-\frac{1}{2}\Delta_{M}u(t,x)=0.$$
  In particular, when $M=\mathbb R^m$
$$p(t,x,y)=\frac{1}{(2\pi t)^{\frac{m}{2}}}e^{-\|x-y\|^2/2t}.$$
 \ \ \ \ Let $X_t$ be the  Brownian motion in $M$ with generator $\Delta_M/2.$ We denote by $\mathbb P_x$ the law of $X_t$ starting from $x\in M,$
 and denote by $\mathbb E_x$ the  expectation with respect to $\mathbb P_x.$

 \noindent\textbf{A. Co-area formula}

  Let $D$ be  a bounded domain with the smooth boundary $\partial D$ in $M.$ Denote by
 $d\pi^{\partial D}_x(y)$  the harmonic measure on $\partial D$ with respect to $x,$ and by
  $g_D(x,y)$ the Green function of $\Delta_M/2$ for  $D$ with  Dirichlet boundary condition and a pole at $x,$ i.e.,
$$-\frac{1}{2}\Delta_{M}g_D(x,y)=\delta_x(y), \ \ y\in D; \ \ \ g_D(x,y)=0, \ \ y\in \partial D.$$
For each $\phi\in \mathscr{C}_{\flat}(D)$
 (space of bounded and continuous functions on $D$),  the \emph{co-area formula} \cite{bass} says that
 \begin{equation}\label{coa}
   \mathbb{E}_x\left[\int_0^{\tau_D}\phi(X_t)dt\right]=\int_{D}g_{D}(x,y)\phi(y)dV(y).
 \end{equation}
 From Proposition 2.8 in \cite{bass},  we note the relation of harmonic measures and hitting times 
 as follows
\begin{equation}\label{hello}
  \mathbb{E}_x\left[\psi(X_{\tau_{D}})\right]=\int_{\partial D}\psi(y)d\pi_x^{\partial D}(y)
\end{equation}
for  $\psi\in\mathscr{C}(\overline{D})$.
  Since $``\mathbb E_x",$  the co-area formula and (\ref{hello}) still work when $\phi, \psi$ are of  a pluripolar set of singularities.

\noindent\textbf{B. It\^o formula}

The following identity is called the \emph{It\^o formula} (see \cite{at,  NN,itoo})
$$u(X_t)-u(x)=B\left(\int_0^t\|\nabla_Mu\|^2(X_s)ds\right)+\frac{1}{2}\int_0^t\Delta_Mu(X_s)dt, \ \ \mathbb P_x-a.s.$$
for  $u\in\mathscr{C}_\flat^2(M)$ (space of bounded $\mathscr{C}^2$-class functions on $M$),
where $B_t$ is the  standard  Brownian motion in $\mathbb R,$ and $\nabla_M$ is the gradient operator on $M$.
 It  follows  the \emph{Dynkin formula}
$$ \mathbb E_x[u(X_T)]-u(x)=\frac{1}{2}\mathbb E_x\left[\int_0^T\Delta_Mu(X_t)dt\right]
$$
for a stopping time $T$ such that each term in the above formula makes sense.
Note that  Dynkin formula still holds for  $u\in\mathscr{C}^2(M)$ if $T=\tau_D.$
In further, it also works when $u$ is of a pluripolar set of singularities, particularly for a plurisubharmonic function $u.$ 

\subsection{Curvatures}~

Let $(M,g)$ be a K\"ahler manifold of complex dimension $m$. 
 We can express the Ricci curvature   of $M$ as ${\rm{Ric}}_M=\sum_{i,j}R_{i\bar{j}}dz_i\otimes d\bar{z}_j,$  where
\begin{equation}\label{syy}
  R_{i\bar{j}}=-\frac{\partial^2}{\partial z_i\partial \bar{z}_j}\log\det(g_{s\bar{t}}).
\end{equation}
A well-known theorem by S. S. Chern asserts  that the \emph{Ricci curvature form}
$$  \mathscr{R}_M:=-dd^c\log\det(g_{s\bar{t}})=\frac{\sqrt{-1}}{2\pi}\sum_{i,j=1}^mR_{i\bar{j}}dz_i\wedge d\bar{z}_j
$$
is a real and closed  (1,1)-form  which represents a cohomology class of  the \emph{de Rham} cohomology group $H^2_{{\rm{DR}}}(M,\mathbb R).$ This cohomology class  depends  only on the complex
structure of $M,$  is called the \emph{first Chern class} of $M.$ Let $s_M$ denote the \emph{Ricci scalar curvature} of $M$ defined by
$$s_M=\sum_{i,j=1}^mg^{i\bar j}R_{i\bar j},$$
where
$(g^{i\bar j})$ is the inverse of $(g_{i\bar j}).$ Invoking  (\ref{syy}), we obtain
$$
  s_M=-\frac{1}{4}\Delta_M\log\det(g_{s\bar t}).
$$
\begin{lemma}\label{s123}  Let $R_M$ be the
pointwise lower bound of Ricci curvature of $M$. Then we have
$$s_M\geq m R_M.$$
\end{lemma}
\begin{proof} Fix any point $x\in M$, we  take a local holomorphic coordinate system $z$ around $x$ such that $g_{i\bar{j}}(x)=\delta^i_j.$
We get
\begin{eqnarray*}
s_M(x)&=&\sum_{j=1}^mR_{j\bar{j}}(x) =\sum_{j=1}^m{\rm{Ric}}_M(\frac{\partial}{\partial z_j},\frac{\partial}{\partial \bar{z}_j})_x\geq mR_M(x)
\end{eqnarray*}
which proves the lemma.
\end{proof}
\section{First Main Theorem}

 We generalize the notions of Nevanlinna's functions to the general K\"ahler manifolds and show a First Main Theorem of meromorphic mapping defined on
  K\"ahler manifolds.
 Let $(M,g)$ be a  K\"ahler manifold of complex dimension $m,$  the associated  K\"ahler  form is defined by
 $$\alpha=\frac{\sqrt{-1}}{2\pi}\sum_{i,j=1}^mg_{i\bar{j}}dz_i\wedge d\bar{z}_j.$$
Fix $o\in M$ as a reference point. 
  Denote by $B_o(r)$ the geodesic ball centered at $o$ with radius $r,$ and by $S_o(r)$ the geodesic sphere centered at $o$ with radius $r.$
 By Sard's theorem, $S_o(r)$ is a submanifold of $M$ for almost all $r>0.$
Also, one denotes by  $g_r(o,x)$  the Green function of $\Delta_M/2$ for $B_o(r)$ with Dirichlet  boundary condition and a pole at $o,$ and by $d\pi_o^r(x)$ the harmonic measure on  $S_o(r)$ with respect to $o.$ 
\subsection{Nevanlinna's functions}~

Let $$f: M\rightarrow N$$
 be a meromorphic mapping to a compact complex manifold $N,$ which means that $f$ is defined by such a
 holomorphic mapping  $f_0:M\setminus I\rightarrow N,$ where $I$ is some analytic  subset of $M$ with $\dim_{\mathbb C}I\leq m-2,$ called the \emph{indeterminacy set} of $f$ such that
the closure  $\overline{G(f_0)}$ of the graph of  $f_0$  is an analytic subset of $M\times N$ and
the natural projection $\overline{G(f_0)}\rightarrow M$ is proper.

 Let a  (1,1)-form $\eta$  on $M,$ we use the following convenient  notation
 $$e_{\eta}(x)=2m\frac{\eta \wedge\alpha^{m-1}}{\alpha^m}.$$
For  an arbitrary  (1,1)-form $\omega$ on $N,$ we define the
  \emph{characteristic function} of $f$ with respect to $\omega$   by
  \begin{eqnarray*}
  T_f(r,\omega)&=&\frac{1}{2}\int_{B_o(r)}g_r(o,x)e_{f^*\omega}(x)dV(x) \\
  &=& \frac{\pi^m}{(m-1)!}\int_{B_o(r)}g_r(o,x)f^*\omega\wedge \alpha^{m-1},
  \end{eqnarray*}
  where $dV$ is the Riemannian volume measure on $M.$
Let a  holomorphic line bundle $L\rightarrow N.$ Equip $L$  with a Hermitian metric $h.$ Since $N$ is compact, 
  we well define $$T_f(r, L): =T_f\big{(}r, c_1(L, h)\big{)}$$
up to a bounded term.

In what follows, we define the \emph{proximity function} and \emph{counting function.}
\begin{lemma}\label{poo}
  $\Delta_M\log (h\circ f)$  is well  defined on $M\setminus I$ and
$$\Delta_M\log(h\circ f)=-4m\frac{f^*c_1(L,h)\wedge\alpha^{m-1}}{\alpha^m}.$$
Hence, we have $$e_{f^*c_1(L,h)}=-\frac{1}{2}\Delta_M\log(h\circ f).$$
\end{lemma}
\begin{proof} Let  $(\{U_\alpha\},\{e_\alpha\})$ be  a local trivialization covering of $(L,h)$ with transition function system $\{g_{\alpha\beta}\}.$ On $U_\alpha\cap U_\beta,$ 
 $$ e_\beta=g_{\alpha\beta}e_\alpha, \ \ h_\alpha=h|_{U_\alpha}=\|e_\alpha\|^2, \ \  h_\beta=h|_{U_\beta}=\|e_\beta\|^2.$$
Thus, we get 
$$\Delta_M\log (h_\beta\circ f)=\Delta_M\log(h_\alpha\circ f)+\Delta_M\log |g_{\alpha\beta}\circ f|^2$$
on  $f^{-1}(U_\alpha\cap U_\beta)\setminus I.$ Notice that $g_{\alpha\beta}$ is holomorphic and nowhere vanishing on
 $U_\alpha\cap U_\beta$, we see that $\log |g_{\alpha\beta}\circ f|^2$ is harmonic on  $f^{-1}(U_\alpha\cap U_\beta)\setminus I.$ So,
$\Delta_M\log (h_\beta\circ f)=\Delta_M\log(h_\alpha\circ f)$ on  $f^{-1}(U_\alpha\cap U_\beta)\setminus I.$ Thus $\Delta_M\log (h\circ f)$  is well  defined on $M\setminus I.$ 
Fix $x\in M,$
we choose a normal holomorphic coordinate system $z$ near $x$ in the sense that  $g_{i\bar{j}}(x)=\delta_{j}^i$ and all the  first-order derivative of $g_{i\overline{j}}$ vanish at $x.$ Then at $x,$ we have
\begin{equation}\label{bbb}
  \Delta_M=4\sum_j\frac{\partial^2}{\partial z_j\partial\bar{z}_j}, \ \ \ \alpha^m=m!\bigwedge_{j=1}^m\frac{\sqrt{-1}}{2\pi}dz_j\wedge d\bar{z}_j
\end{equation}
and 
$$f^*c_1(L,h)\wedge\alpha^{m-1}=-(m-1)!{\rm{tr}}\left(\frac{\partial^2\log(h\circ f)}{\partial z_i\partial\bar{z}_j}\right)\bigwedge_{j=1}^m\frac{\sqrt{-1}}{2\pi}dz_j\wedge d\bar{z}_j,$$
where ``tr" stands for the trace of a square matrix. Indeed, by (\ref{bbb}) 
$$\Delta_M\log(h\circ f)=4{\rm{tr}}\left(\frac{\partial^2\log(h\circ f)}{\partial z_i\partial\bar{z}_j}\right)$$
 at $x.$  This proves the lemma.
\end{proof}
Let $s\in H^0(N, L)$ which is not equal to 0. Locally, we  write $s=\tilde s e,$ where $e$ is a local holomorphic frame of $L.$ Then 
$$
  \Delta_M\log \|s\circ f\|^2=\Delta_M\log(h\circ f)+\Delta_M\log|\tilde{s}\circ f|^2.
$$
 By  the similar arguments as in the proof of Lemma \ref{poo},
 we get
$$  \Delta_M\log|\tilde{s}\circ f|^2=4m\frac{dd^c\log|\tilde{s}\circ f|^2\wedge\alpha^{m-1}}{\alpha^m}.
$$
 \begin{lemma}\label{inte}  Let $s\in H^0(N,L)$ with $(s)=D.$  If $(L,h)\geq0,$ then

  {\rm{(i)}} $\log\|s\circ f\|^2$ is locally the difference of two plurisubharmonic functions, and hence $\log\|s\circ f\|^2\in \mathscr{L}_{loc}(M).$

  {\rm{(ii)}} $dd^c\log\|s\circ f\|^2=f^*D-f^*c_1(L,h)$ in the sense of currents.
\end{lemma}
\begin{proof} Locally, we can write $s=\tilde se,$ where $e$ is a local holomorphic frame of $L$ with $h=\|e\|^2.$
Then 
$$\log\|s\circ f\|^{2}=\log|\tilde{s}\circ f|^{2}+\log (h\circ f).$$ 
Since  $c_1(L,h)\geq0,$ one obtains  $-dd^c\log (h\circ f)\geq0.$ Indeed, $\tilde{s}$ is holomorphic,   hence  $dd^c\log|\tilde{s}\circ f|^{2}\geq0.$
 This follows (i).  Poincar\'e-Lelong formula implies that
$dd^c\log|\tilde{s}\circ f|^2=f^*D$
in the sense of currents,  hence (ii) holds.
\end{proof}
Assume that $L\geq0.$ 
The \emph{proximity function} of $f$ with respect to $D\in|L|$ is  defined by
$$  m_f(r,D)=\int_{S_o(r)}\log\frac{1}{\|s_D\circ f(x)\|}d\pi_o^r(x).
$$
Write
$$\log\|s_D\circ f\|^{-2}=\log (h\circ f)^{-1}-\log|\tilde{s}_{D}\circ f|^{2}$$
as  the difference of two plurisubharmonic functions.
Then it defines a Riesz charge $d\mu=d\mu_1-d\mu_2,$ here $d\mu_2$  gives a Riesz measure for $f^*D.$  The \emph{counting function} of $f$ with respect to $D$ is  defined by
\begin{eqnarray*}
   N_f(r,D)&=& \frac{1}{4}\int_{B_o(r)}g_r(o,x)\Delta_M\log|\tilde{s}_D\circ f(x)|^2dV(x) \\
   &=&\frac{\pi^m}{(m-1)!}\int_{B_o(r)}g_r(o,x)dd^c\log|\tilde{s}_D\circ f|^2\wedge\alpha^{m-1} \\
&=&\frac{\pi^m}{(m-1)!}\int_{f^*D\cap B_o(r)}g_r(o,x)\alpha^{m-1}.
\end{eqnarray*}
Similarly, one can define $N(r,{\rm{Supp}}f^*D).$ We write  $\overline{N}_f(r,D)=N(r, {\rm{Supp}}f^*D)$ in short.
\subsection{Probabilistic expressions of  Nevanlinna's functions}~

We  reformulate  Nevanlinna's functions in terms of Brownian motion $X_t$. Let  $I$ be the indeterminacy set of $f.$
Set the stopping time $$\tau_r=\inf\big{\{}t>0: X_t\not\in B_o(r)\big{\}}.$$
Put $\omega:=-dd^c\log h.$  By co-area formula, we have 
$$T_f(r,L) =\frac{1}{2}\mathbb E_o\left[\int_0^{\tau_r}e_{f^*\omega}(X_t)dt\right].$$
The relation between  harmonic measures and hitting times gives that
$$  m_f(r,D) =\mathbb E_o\left[\log\frac{1}{\|s_D\circ f(X_{\tau_r})\|}\right].$$
To counting function $N_f(r,D),$ we use an alternative probabilistic expression (see \cite{at,atsuji,carne})  as follows
\begin{equation}\label{prob}
  N_f(r,D)=\lim_{\lambda\rightarrow\infty}\lambda\mathbb P_o\left(\sup_{0\leq t\leq\tau_r}\log\frac{1}{\|s_D\circ f(X_t)\|}>\lambda\right).
\end{equation}
Following  the arguments in \cite{worthy} related to the local martingales, we see that the  above limit exists.
By Dynkin formula  and co-area formula, it equals
\begin{eqnarray*}
  &&  \lim_{\lambda\rightarrow\infty}\lambda\mathbb P_o\left(\sup_{0\leq t\leq\tau_r}\log\frac{1}{\|s_D\circ f(X_t)\|}>\lambda\right) \\
  &=& -\frac{1}{2}\mathbb E_o\left[\int_0^{\tau_r}\Delta_M\log\frac{1}{|\tilde{s}_D\circ f(X_t)|}dt\right] \\
    &=& \frac{1}{4}\int_{B_o(r)}g_r(o,x)\Delta_M\log|\tilde{s}_D\circ f(x)|^2dV(x) \\
  &=& N_f(r,D).
\end{eqnarray*}
\begin{remark}\label{remark} The definitions of Nevanlinna's functions in above are  natural extensions of the classical ones. To see that, we recall the  $\mathbb C^m$-case:
\begin{eqnarray*}
   T_f(r,L)&=&\int_0^r\frac{dt}{t^{2m-1}}\int_{B_o(t)}f^*c_1(L,h)\wedge\alpha^{m-1}, \\
  m_f(r,D)&=&\int_{S_o(r)}\log\frac{1}{\|s_D\circ f\|}\gamma, \\
 N_f(r,D)&=&\int_0^r\frac{dt}{t^{2m-1}}\int_{B_o(t)}dd^c\log|\tilde{s}_D\circ f|^2\wedge\alpha^{m-1},
\end{eqnarray*}
where $o=(0,\cdots,0)$ and
$$\alpha=dd^c\|z\|^2,\ \ \ \gamma=d^c\log\|z\|^2\wedge \left(dd^c\log\|z\|^2\right)^{m-1}.$$
Notice  the facts that
$$\gamma=d\pi_o^r(z),  \ \ \ g_r(o,z)=\left\{
                \begin{array}{ll}
                  \frac{\|z\|^{2-2m}-r^{2-2m}}{(m-1)\omega_{2m-1}}, & m\geq2; \\
                  \frac{1}{\pi}\log\frac{r}{|z|}, & m=1.
                \end{array}
              \right.,$$
where $\omega_{2m-1}$ is the volume of  unit sphere in $\mathbb R^{2m}.$ Apply integration by part, we see  it coincides with  ours.
\end{remark}
\subsection{First Main Theorem}~

 Let $N$ be a complex projective algebraic manifold. There is  a very ample holomorphic line bundle $L'\rightarrow V.$
 Equip $L'$ with a Hermitian metric $h'$ such that  $\omega':=-dd^c\log h'>0.$ 
 For an arbitrary  holomorphic line bundle $L\rightarrow N$ equipped with a Hermitian metric $h,$
   whose  Chern form says $\omega:=-dd^c\log h,$ we can pick  $k\in \mathbb N$ large enough so that
$\omega+k\omega'>0.$
Take the natural product Hermitian metric $\|\cdot\|$ on $L\otimes L'^{\otimes k},$ then the Chern form is $\omega+k\omega'$. 
Choose $\sigma\in H^0(M,L')$ such that $f(M)\not\subset{\rm{Supp}}(\sigma).$ 
Due to $\omega+k\omega'>0$ and $\omega'>0,$  we see that $\log\|(s_D\otimes \sigma^k)\circ f\|^2$ and $\log\|\sigma\circ f\|^2$ are
locally the difference of two plurisubharmonic functions, where $D\in|L|$. Thus,
$$\log\|s_D\circ f\|^2=\log\|(s_D\otimes \sigma^k)\circ f\|^2-k\log\|\sigma\circ f\|^2$$
 is locally the difference of two plurisubharmonic functions. Namely,  $m_f(r,D)$ can be  defined. 

We have the First Main Theorem (FMT):
\begin{theorem}[FMT]\label{first2}  We have
$$T_f(r,L)=m_f(r,D)+N_f(r,D)+O(1).$$
\end{theorem}
\begin{proof}  
Set
$$T_{\lambda,r}=\inf\Big{\{}t>0: \sup_{s\in[0,t]\setminus T_{I,r}}\log\frac{1}{\|s_D\circ f(X_s)\|}>\lambda\Big{\}},$$
where $T_{I,r}=\{0\leq t\leq\tau_r: X_t\in I\}$ and $I$ is the indeterminacy set of $f.$
Due to  the definition of $T_{\lambda,r},$  $X_t$ does not pass through  $f^*D$ as well as  those points in $I$ near which $\log\|s_D\circ f(X_t)\|^{-1}$ is unbounded when  $0\leq t\leq \tau_r\wedge T_{\lambda,r}.$ By Dynkin formula, it follows that
\begin{eqnarray}\label{cvc}
   && \mathbb E_o\left[\log\frac{1}{\|s_D\circ f(X_{\tau_r\wedge T_{\lambda,r}})\|}\right]  \\
&=& \frac{1}{2}\mathbb E_o\left[\int_0^{\tau_r\wedge T_{\lambda,r}}\Delta_M
\log\frac{1}{\|s_D\circ f(X_{t})\|}dt\right]+\log\frac{1}{\|s_D\circ f(o)\|} \nonumber,
\end{eqnarray}
where $\tau_r\wedge T_{\lambda,r}=\min\{\tau_r, T_{\lambda,r}\}.$ Note that $\Delta_M\log|\tilde{s}_D\circ f|=0$ outside $I\cup f^*D,$  we see that
$$\Delta_M
\log\frac{1}{\|s_D\circ f(X_{t})\|}=-\frac{1}{2}\Delta_M\log h\circ f(X_{t})$$
for $t\in [0,T_{\lambda,r}].$
Thus, (\ref{cvc}) turns to
$$ \mathbb E_o\left[\log\frac{1}{\|s_D\circ f(X_{\tau_r\wedge T_{\lambda,r}})\|}\right] \\=
  -\frac{1}{4}\mathbb E_o\left[\int_0^{\tau_r\wedge T_{\lambda,r}}\Delta_M\log h\circ f(X_{t})dt\right]
  +O(1).
$$
The monotone convergence theorem leads to
\begin{equation*}
 \frac{1}{4}\mathbb E_o\left[\int_0^{\tau_r\wedge T_{\lambda,r}}\Delta_M\log h\circ f(X_{t})dt\right]
    \rightarrow\frac{1}{2}\mathbb E_o\left[\int_0^{\tau_r}e_{f^*\omega}(X_t)dt\right]
    =T_f(r,L)
\end{equation*}
as $\lambda\rightarrow\infty,$ due to  $T_{\lambda,r}\rightarrow\infty$ $a.s.$ as $\lambda\rightarrow\infty.$
We handle the first term  in (\ref{cvc}),  write it as two parts:
\begin{eqnarray*}
   \mathrm{I}+\mathrm{II}&=& \mathbb E_o\left[\log\frac{1}{\|s_D\circ f(X_{\tau_r})\|}: \tau_r<T_{\lambda,r}\right]
    +
\mathbb E_o\left[\log\frac{1}{\|s_D\circ f(X_{T_{\lambda,r}})\|}: T_{\lambda,r}\leq\tau_r \right].
\end{eqnarray*}
By the monotone convergence theorem again, 
\begin{equation*}
  \mathrm{I}\rightarrow \mathbb E_o\left[\log\frac{1}{\|s_D\circ f(X_{\tau_r})\|}\right]= m_f(r,D)
\end{equation*}
as $\lambda\rightarrow\infty.$
Finally, we deal with  $\mathrm{II}.$ By the definition of $T_{\lambda, r},$ we see that
\begin{eqnarray*}
  \mathrm{II}&=&
 \lambda\mathbb P_o\left(\sup_{t\in[0,\tau_r]\setminus T_{I,r}}\log\frac{1}{\|s_D\circ f(X_t)\|}>\lambda\right)\rightarrow N_f(r,D)
\end{eqnarray*}
as $\lambda\rightarrow\infty.$ Put  the above together,
we show the theorem.
\end{proof}

\begin{cor}[Nevanlinna inequality]\label{bzda} We have
$$N_f(r,D)\leq T_f(r,L)+O(1).$$
\end{cor}

\section{Logarithmic Derivative Lemma}
The goals of this section are to prove the  Logarithmic Derivative Lemma  for K\"ahler manifolds (i.e., Theorem \ref{log1}) and provide an estimate of $C(o,r,\delta)$ when the K\"ahler manifolds are non-positively curve.
The Logarithmic Derivative Lemma plays an useful role in derivation of the Second Main Theorem in  Section 5.

\subsection{Logarithmic Derivative Lemma}~

 Let $(M,g)$ be a $m$-dimensional complete K\"ahler manifold,  and  $\nabla_M$  be the gradient operator on $M$  associated with  $g.$
 Let $X_t$ be the Brownian motion in $M$ with generator $\Delta_M/2.$

We first prepare some lemmas:
\begin{lemma}[Calculus Lemma, \cite{at}]\label{calculus} Let $k\geq0$ be a
locally integrable function  on $M$ such that it is locally bounded at $o\in M.$
Then for any $\delta>0,$ there exists a function  $C(o, r, \delta)>0$  independent of $k$ and a set
$E_{\delta}\subset [0, \infty)$ of finite Lebesgue measure such that
\begin{equation}\label{C}
  \mathbb E_o\big{[}k(X_{\tau_r})\big{]}\leq C(o, r, \delta)
\left(\mathbb E_o\left[\int_0^{\tau_r} k(X_t)dt\right]\right)^{(1+\delta)^2}
\end{equation}
holds for $r>1$  outside $E_{\delta}.$
\end{lemma}

Let $\psi$ be a meromorphic function on $M.$
The norm of the gradient of $\psi$ is defined by
$$\|\nabla_M\psi\|^2=\sum_{i,j}g^{i\overline j}\frac{\partial\psi}{\partial z_i}\overline{\frac{\partial \psi}{\partial  z_j}},$$
where $(g^{i\overline{j}})$ is the inverse of $(g_{i\overline{j}}).$
Locally, we write $\psi=\psi_1/\psi_0,$ where $\psi_0,\psi_1$ are holomorphic functions so that ${\rm{codim}}_{\mathbb C}(\psi_0=\psi_1=0)\geq2$ if $\dim_{\mathbb C}M\geq2.$ Identify $\psi$  with a meromorphic mapping into $\mathbb P^1(\mathbb C)$  by
$x\mapsto[\psi_0(x):\psi_1(x)].$ The characteristic function of $\psi$ with respect to the Fubini-Study form $\omega_{FS}$ on  $\mathbb P^1(\mathbb C)$  is defined by
$$T_\psi(r,\omega_{FS})=\frac{1}{4}\int_{B_o(r)}g_r(o,x)\Delta_M\log(|\psi_0(x)|^2+|\psi_1(x)|^2)dV(x).$$
Let
  $i:\mathbb C\hookrightarrow\mathbb P^1(\mathbb C)$ be an inclusion defined by
 $z\mapsto[1:z].$  Via the pull-back by $i,$ we have a (1,1)-form $i^*\omega_{FS}=dd^c\log(1+|\zeta|^2)$ on $\mathbb C,$
 where $\zeta:=w_1/w_0$ and $[w_0:w_1]$ is
the homogeneous coordinate system of $\mathbb P^1(\mathbb C).$ The characteristic function of $\psi$ with respect to $i^*\omega_{FS}$ is defined by
$$\hat{T}_\psi(r,\omega_{FS}) = \frac{1}{4}\int_{B_o(r)}g_r(o,x)\Delta_M\log(1+|\psi(x)|^2)dV(x).$$
Clearly, $$\hat{T}_\psi(r,\omega_{FS})\leq T_\psi(r,\omega_{FS}).$$
We adopt the spherical distance $\|\cdot,\cdot\|$ on  $\mathbb P^1(\mathbb C),$ the proximity function of $\psi$  with respect to
$a\in \mathbb P^1(\mathbb C)=\mathbb C\cup\{\infty\}$
is defined by
$$\hat{m}_\psi(r,a)=\int_{S_o(r)}\log\frac{1}{\|\psi(x),a\|}d\pi_o^r(x).$$
Again,  set
$$\hat{N}_\psi(r,a)=\frac{\pi^m}{(m-1)!}\int_{\psi^{-1}(a)\cap B_o(r)}g_r(o,x)\alpha^{m-1}.$$
Using the similar arguments as in the proof of Theorem \ref{first2}, we  easily show that
$\hat{T}_\psi(r,\omega_{FS})=\hat{m}_\psi(r,a)+\hat{N}_\psi(r,a)+O(1).$
We  also define the Nevanlinna's characteristic function
$$T(r,\psi):=m(r,\psi)+N(r,\psi),$$
where
\begin{eqnarray*}
m(r,\psi)&=&\int_{S_o(r)}\log^+|\psi(x)|d\pi^r_o(x), \\
N(r,\psi)&=& \frac{\pi^m}{(m-1)!}\int_{\psi^{-1}(\infty)\cap B_o(r)}g_r(o,x)\alpha^{m-1}.
 \end{eqnarray*}
 Clearly,  $N(r,\psi)=\hat{N}_\psi(r,\infty)$ and $m(r,\psi)=\hat{m}_\psi(r,\infty)+O(1).$ Thus,
 \begin{equation}\label{goed}
   T(r,\psi)=\hat{T}_\psi(r,\omega_{FS})+O(1), \ \ \ T\Big(r,\frac{1}{\psi-a}\Big)= T(r,\psi)+O(1).
 \end{equation}
 \ \ \ \ On $\mathbb P^1(\mathbb C),$ we take a singular metric
$$\Phi=\frac{1}{|\zeta|^2(1+\log^2|\zeta|)}\frac{\sqrt{-1}}{4\pi^2}d\zeta\wedge d\bar \zeta.$$
A direct computation shows that
\begin{equation}\label{ada}
\int_{\mathbb P^1(\mathbb C)}\Phi=1, \ \ \ 2m\pi\frac{\psi^*\Phi\wedge\alpha^{m-1}}{\alpha^m}=\frac{\|\nabla_M\psi\|^2}{|\psi|^2(1+\log^2|\psi|)}.
\end{equation}
Set
$$T_\psi(r,\Phi)=\frac{1}{2}\int_{B_o(r)}g_r(o,x)e_{\psi^*\Phi}(x)dV(x).$$
Invoking (\ref{ada}), we obtain
\begin{equation}\label{ffww}
  T_\psi(r,\Phi)=\frac{1}{2\pi}\int_{B_o(r)}g_r(o,x)\frac{\|\nabla_M\psi\|^2}{|\psi|^2(1+\log^2|\psi|)}(x)dV(x).
\end{equation}
\begin{lemma}\label{oo12} We have
$$T_\psi(r,\Phi)\leq T(r,\psi)+O(1).$$
\end{lemma}
\begin{proof} By Fubini's theorem and Corollary \ref{bzda}
\begin{eqnarray*}
T_\psi(r,\Phi)&=&m\int_{B_o(r)}g_r(o,x)\frac{\psi^*\Phi\wedge\alpha^{m-1}}{\alpha^m}dV(x)  \\
&=&\frac{\pi^m}{(m-1)!}\int_{\zeta\in\mathbb P^1(\mathbb C)}\Phi(\zeta)\int_{\psi^{-1}(\zeta)\cap B_o(r)}g_r(o,x)\alpha^{m-1} \\
&=&\int_{\zeta\in\mathbb P^1(\mathbb C)}N_\psi(r,\zeta)\Phi(\zeta)  \\
&\leq&\int_{\zeta\in\mathbb P^1(\mathbb C)}\big{(}T(r,\psi)+O(1)\big{)}\Phi(\zeta) \\
&=& T(r,\psi)+O(1).
\end{eqnarray*}
The proof is completed.
\end{proof}
\begin{lemma}\label{999a} Assume that $\psi(x)\not\equiv0.$ For any $\delta>0,$ there are $C(o, r, \delta)>0$  independent of $\psi$ and
 $E_\delta\subset(1,\infty)$ of finite Lebesgue measure such that
\begin{eqnarray*}
  \mathbb E_o\left[\log^+\frac{\|\nabla_M\psi\|^2}{|\psi|^2(1+\log^2|\psi|)}(X_{\tau_r})\right]
  &\leq&(1+\delta)^2\log T(r,\psi)+\log C(o, r, \delta)
\end{eqnarray*}
 holds for $r>1$ outside  $E_\delta.$
\end{lemma}
\begin{proof} By Jensen inequality, it is clear that
\begin{eqnarray*}
  \mathbb E_o\left[\log^+\frac{\|\nabla_M\psi\|^2}{|\psi|^2(1+\log^2|\psi|)}(X_{\tau_r})\right]
   &\leq&  \mathbb E_o\left[\log\Big{(}1+\frac{\|\nabla_M\psi\|^2}{|\psi|^2(1+\log^2|\psi|)}(X_{\tau_r})\Big{)}\right] \nonumber \\
    &\leq& \log^+\mathbb E_o\left[\frac{\|\nabla_M\psi\|^2}{|\psi|^2(1+\log^2|\psi|)}(X_{\tau_r})\right]+O(1). \nonumber
\end{eqnarray*}
By Lemma \ref{calculus} and co-area formula, there is $C(o, r, \delta)>0$  such that
\begin{eqnarray*}
   && \log^+\mathbb E_o\left[\frac{\|\nabla_M\psi\|^2}{|\psi|^2(1+\log^2|\psi|)}(X_{\tau_r})\right]  \\
   &\leq& (1+\delta)^2\log^+\mathbb E_o\left[\int_0^{\tau_r}\frac{\|\nabla_M\psi\|^2}{|\psi|^2(1+\log^2|\psi|)}(X_{t})dt\right]
   +\log C(o, r, \delta)
   \nonumber \\
   &\leq& (1+\delta)^2\log T(r,\psi)+\log C(o, r, \delta)+O(1), \nonumber
\end{eqnarray*}
where  Lemma \ref{oo12} and (\ref{ffww}) are applied. Modify $C(o, r, \delta)$ such that the term $O(1)$ is removed, then we get the desired  inequality.
\end{proof}

Define
$$m\left(r,\frac{\|\nabla_M\psi\|}{|\psi|}\right)=\int_{S_o(r)}\log^+\frac{\|\nabla_M\psi\|}{|\psi|}(x)d\pi^r_o(x).$$

We now prove Theorem \ref{log1}:

\begin{proof} On the one hand,
\begin{eqnarray*}
    m\left(r,\frac{\|\nabla_M\psi\|}{|\psi|}\right)
   &\leq& \frac{1}{2}\int_{S_o(r)}\log^+\frac{\|\nabla_M\psi\|^2}{|\psi|^2(1+\log^2|\psi|)}(x)d\pi^r_o(x) \\
  && +\frac{1}{2}\int_{S_o(r)}\log^+\big{(}1+\log^2|\psi(x)|\big{)}d\pi^r_o(x) \\
  &=& \frac{1}{2}\mathbb E_o\left[\log^+\frac{\|\nabla_M\psi\|^2}{|\psi|^2(1+\log^2|\psi|)}(X_{\tau_r})\right] \\
   && +\frac{1}{2}\int_{S_o(r)}\log\big{(}1+\log^2|\psi(x)|\big{)}d\pi^r_o(x)  \\
    &\leq& \frac{1}{2}\mathbb E_o\left[\log^+\frac{\|\nabla_M\psi\|^2}{|\psi|^2(1+\log^2|\psi|)}(X_{\tau_r})\right] \\
  && +\frac{1}{2}\int_{S_o(r)}\log\Big{(}1+\big{(}\log^+|\psi(x)|+\log^+\frac{1}{|\psi(x)|}\big{)}^2\Big{)}d\pi^r_o(x)  \\
   &\leq& \frac{1}{2}\mathbb E_o\left[\log^+\frac{\|\nabla_M\psi\|^2}{|\psi|^2(1+\log^2|\psi|)}(X_{\tau_r})\right]  \\
  && +\int_{S_o(r)}\log\Big{(}1+\log^+|\psi(x)|+\log^+\frac{1}{|\psi(x)|}\Big{)}d\pi^r_o(x).
\end{eqnarray*}
Lemma \ref{999a} implies that
\begin{eqnarray*}
&& \frac{1}{2}\mathbb E_o\left[\log^+\frac{\|\nabla_M\psi\|^2}{|\psi|^2(1+\log^2|\psi|)}(X_{\tau_r})\right] \\
 &\leq& \frac{(1+\delta)^2}{2}\log T(r,\psi)+\frac{1}{2}\log C(o, r, \delta)+O(1).
\end{eqnarray*}
On the other hand, by Jensen inequality  and (\ref{goed})
\begin{eqnarray*}
&& \int_{S_o(r)}\log\Big{(}1+\log^+|\psi(x)|+\log^+\frac{1}{|\psi(x)|}\Big{)}d\pi^r_o(x) \\
 &\leq& \log\int_{S_o(r)}\Big{(}1+\log^+|\psi(x)|+\log^+\frac{1}{|\psi(x)|}\Big{)}d\pi^r_o(x) \\
   &\leq& \log \big{(}m(r,\psi) +m(r,1/\psi)\big{)}+O(1)  \\
   &\leq& \log T(r,\psi)+O(1).
\end{eqnarray*}
Replacing $C(o, r, \delta)$ by $C^2(o, r, \delta)$ and combining the above, then the theorem can be proved.
\end{proof}

\subsection{Estimate of  $C(o, r, \delta)$}~

Let $M$ be a complete K\"ahler manifold of non-positive sectional curvature.
Indeed, we let $\kappa$ be defined by (\ref{kappa}). Clearly, $\kappa$ is a non-positive, non-increasing and continuous function  on $[0,\infty).$
   Treat the  differential equation
 \begin{equation}\label{G}
G''(t)+\kappa(t)G(t)=0;\ \ \ G(0)=0, \ \ G'(0)=1
 \end{equation}
on $[0,\infty).$  Now compare (\ref{G})  with $y''(t)+\kappa(0)y(t)=0$ under the same  initial conditions, we see that
 $G$ can be  estimated simply as
$$G(t)=t \ \ \text{for}  \ \kappa\equiv0; \ \ \ G(t)\geq t \ \ \text{for} \ \kappa\not\equiv0.$$
This follows that
\begin{equation}\label{v1}
  G(r)\geq r \ \ \text{for} \ r\geq0; \ \ \ \int_1^r\frac{dt}{G(t)}\leq\log r \ \ \text{for} \ r\geq1.
\end{equation}
On the other hand, we  rewrite (\ref{G}) as the form
$$\log'G(t)\cdot\log'G'(t)=-\kappa(t).$$
Since $G(t)\geq t$ is increasing,
then the decrease and non-positivity of $\kappa$ imply that for each fixed $t,$ $G$  must be satisfied one of the following two inequalities
$$\log'G(t)\leq\sqrt{-\kappa(t)} \ \ \text{for} \ t>0; \ \ \ \log'G'(t)\leq\sqrt{-\kappa(t)} \ \ \text{for} \ t\geq0.$$
By virtue of $G(t)\rightarrow0$ as $t\rightarrow0,$ by integration, $G$ is bounded from above by
\begin{equation}\label{v2}
  G(r)\leq r\exp\big(r\sqrt{-\kappa(r)}\big) \ \  \text{for} \ r\geq0.
\end{equation}
\ \ \ \ In what follows, one assumes that $M$ is simply connected.
The  purpose  of this section is to show the following Logarithmic Derivative Lemma (LDL) by estimating $C(o, r, \delta).$
\begin{theorem}[LDL]\label{log} Let
$\psi$ be a nonconstant meromorphic function on $M$. Then
\begin{eqnarray*}
   m\left(r,\frac{\|\nabla_M\psi\|}{|\psi|}\right)&\leq& \Big{(}1+\frac{(1+\delta)^2}{2}\Big{)}\log T(r,\psi)
   +O\Big{(}r\sqrt{-\kappa(r)}+\delta\log r\Big{)}  \  \big{\|},
\end{eqnarray*}
where $\kappa$ is defined by $(\ref{kappa}).$
\end{theorem}

\textbf{Remark.}  The LDL still holds when $M$ is multi-connected,  one just needs to  lift $M$ to the universal covering, see the arguments in Section 5.3.

We first introduce some lemmas.  
\begin{lemma}[\cite{atsuji}]\label{zz} Let $\eta>0$ be a  number. Then there is a constant $C>0$ such that
  $$g_r(o,x)\int_{\eta}^rG^{1-2m}(t)dt\geq C\int_{r(x)}^rG^{1-2m}(t)dt$$
  holds for $r>\eta$ and $x\in B_o(r)\setminus \overline{B_o(\eta)},$ where $G$ is defined by {\rm{(\ref{G})}}.
\end{lemma}

\begin{lemma}[\cite{Deb,13}]\label{sing} Let $M$ be a simply-connected, non-positively curved and complete Hermitian manifold of complex dimension $m$. Then
$$ (i) \ \ \   g_r(o,x)\leq\left\{
                \begin{array}{ll}
                  \frac{1}{\pi}\log\frac{r}{r(x)}, & m=1 \\
                  \frac{1}{(m-1)\omega_{2m-1}}\big{(}r^{2-2m}(x)-r^{2-2m}\big{)}, & m\geq2 \\
                \end{array}
              \right.;  \ \ \ \  \  \  \ \ \ \ \  $$
$$(ii) \ \ \ d\pi^r_{o}(x)\leq\frac{1}{\omega_{2m-1}r^{2m-1}}d\sigma_r(x), \ \ \ \ \ \ \ \  \  \ \ \ \ \ \ \ \ \ \ \ \  \ \ \ \ \
\  \ \ \ \ \ \ \ \ \ \ \ \  $$
 where $g_r(o,x)$ denotes  the Green function of $\Delta_M/2$ for $B_o(r)$ with  Dirichlet boundary condition and a pole at $o,$ 
   and $d\pi_{o}^r(x)$ is the harmonic measure on $S_o(r)$ with respect to $o,$ and $\omega_{2m-1}$ is the Euclidean volume of unit sphere in $\mathbb R^{2m},$ and
  $d\sigma_r(x)$ is the induced Riemannian volume measure on $S_o(r).$
\end{lemma}

\begin{lemma}[Borel Lemma, \cite{ru}]\label{cal1} Let $T$ be a strictly positive nondecreasing  function of $\mathscr{C}^1$-class on $(0,\infty).$ Let $\gamma>0$ be a number such that $T(\gamma)\geq e,$ and $\phi$ be a strictly positive nondecreasing function such that
$$c_\phi=\int_e^\infty\frac{1}{t\phi(t)}dt<\infty.$$
Then, the inequality
  $$T'(r)\leq T(r)\phi(T(r))$$
holds for  $r\geq\gamma$ outside a set of Lebesgue measure not exceeding $c_\phi.$ Particularly, take $\phi(T)=T^\delta$ for a number $\delta>0,$  we have
  $T'(r)\leq T^{1+\delta}(r)$
holds for  $r>0$ outside a set $E_\delta\subset(0,\infty)$ of finite Lebesgue measure.
\end{lemma}
We now prove the following so-called Calculus Lemma (see also \cite{atsuji}) which gives an estimate of $C(o,r,\delta).$
\begin{lemma}[Calculus Lemma]\label{calculusII} Let $k\geq0$ be a locally integrable function on $M$ such that it is locally bounded at $o\in M.$
 Then for any $\delta>0,$ there is a constant $C>0$ independent of $k,\delta,$ and a set $E_\delta\subset(1,\infty)$ of finite Lebesgue measure such that
  $$\mathbb E_o[k(X_{\tau_r})]\leq \frac{C^{(1+\delta)^2}
\log^{(1+\delta)^2}r}{r^{(1-2m)\delta}e^{(1-2m)(1+\delta)r\sqrt{-\kappa(r)}}}\left(\mathbb E_o\left[\int_0^{\tau_r}k(X_{t})dt\right]\right)^{(1+\delta)^2}$$
  holds for $r>1$ outside $E_\delta,$ where $\kappa$ is defined by $(\ref{kappa}).$
\end{lemma}
\begin{proof}
By  Lemma \ref{zz} and Lemma \ref{sing} with (\ref{v1}), we get
\begin{eqnarray*}
   \mathbb E_o\left[\int_0^{\tau_{r}}k(X_{t})dt\right]&=& \int_{B_o(r)}g_r(o,x)k(x)dV(x) \\
   &=&\int_0^rdt\int_{S_o(t)}g_r(o,x)k(x)d\sigma_t(x)  \\
&\geq& C_0\int_0^r\frac{\int_t^rG^{1-2m}(s)ds}{\int_1^rG^{1-2m}(s)ds}dt\int_{S_o(t)}k(x)d\sigma_t(x) \\
&=& \frac{C_0}{\log r}\int_0^rdt\int_t^rG^{1-2m}(s)ds\int_{S_o(t)}k(x)d\sigma_t(x)
\end{eqnarray*}
and
$$\mathbb E_o\big{[}k(X_{\tau_r})\big{]}=\int_{S_o(r)}k(x)d\pi_o^r(x)\leq\frac{1}{\omega_{2m-1}r^{2m-1}}\int_{S_o(r)}k(x)d\sigma_r(x),$$
where $\omega_{2m-1}$ denotes the Euclidean volume of unit sphere in $\mathbb R^{2m},$
  $d\sigma_r$ is the induced volume measure  on $S_o(r).$ Hence,
 $$\mathbb E_o\left[\int_0^{\tau_{r}}k(X_{t})dt\right]\geq\frac{C_0}{\log r}\int_0^rdt\int_t^rG^{1-2m}(s)ds\int_{S_o(t)}k(x)d\sigma_t(x)$$
and
\begin{equation}\label{fr}
  \mathbb E_o\big{[}k(X_{\tau_r})\big{]}\leq\frac{1}{\omega_{2m-1}r^{2m-1}}\int_{S_o(r)}k(x)d\sigma_r(x).
\end{equation}
 Put
$$\Gamma(r)=\int_0^rdt\int_t^rG^{1-2m}(s)ds\int_{S_o(t)}k(x)d\sigma_t(x).$$
Then
\begin{equation}\label{A1}
 \Gamma(r)\leq\frac{\log r}{C_0}\mathbb E_o\left[\int_0^{\tau_{r}}k(X_{t})dt\right].
\end{equation}
A simple computation shows that
$$\Gamma'(r)=G^{1-2m}(r)\int_0^rdt\int_{S_o(t)}k(x)d\sigma_t(x).$$
By this with (\ref{fr})
\begin{equation}\label{B2}
  \mathbb E_o\big{[}k(X_{\tau_r})\big{]}\leq\frac{1}{\omega_{2m-1}r^{2m-1}}\frac{d}{dr}\left(\frac{\Gamma'(r)}{G^{1-2m}(r)}\right).
\end{equation}
Using Lemma \ref{cal1} twice, for any $\delta>0$ we have
\begin{equation}\label{C3}
  \frac{d}{dr}\left(\frac{\Gamma'(r)}{G^{1-2m}(r)}\right)\leq\frac{\Gamma^{(1+\delta)^2}(r)}{G^{(1-2m)(1+\delta)}(r)}
\end{equation}
holds outside a set $E_{\delta}\subset(1,\infty)$ of finite Lebesgue measure.
Using  (\ref{A1})-(\ref{C3}) and (\ref{v2}), it is not hard to conclude that
$$\mathbb E_o\big{[}k(X_{\tau_r})\big{]}\leq \frac{C^{(1+\delta)^2}
\log^{(1+\delta)^2}r}{r^{(1-2m)\delta}e^{(1-2m)(1+\delta)r\sqrt{-\kappa(r)}}}\left(\mathbb E_o\left[\int_0^{\tau_{r}}k(X_{t})dt\right]\right)^{(1+\delta)^2}$$
with $C=1/C_0>0$ being a constant independent of $k,\delta.$
\end{proof}
Lemma \ref{calculusII} implies an estimate
$$ C(o, r, \delta)\leq \frac{C^{(1+\delta)^2}
\log^{(1+\delta)^2}r}{r^{(1-2m)\delta}e^{(1-2m)(1+\delta)r\sqrt{-\kappa(r)}}}.$$
Thus, we get
\begin{equation}\label{esti}
  \log C(o, r, \delta)\leq O\Big{(}r\sqrt{-\kappa(r)}+\delta\log r\Big{)}.
\end{equation}

We prove Theorem \ref{log}:
\begin{proof}
Combining Theorem \ref{log1} with (\ref{esti}), we show the theorem.
\end{proof}

\section{Second Main Theorem}

\subsection{Meromorphic mappings into $\mathbb P^n(\mathbb C)$}~

In this subsection,  $M$ is a general  K\"ahler manifold.

 Let $\psi:M\rightarrow \mathbb P^n(\mathbb C)$ be a  meromorphic mapping, i.e.,  there exists an open covering $\{U_\alpha\}$ of $M$ such that
   $\psi$ has a local representation
 $[\psi_0^{\alpha}:\cdots:\psi_{n}^\alpha]$
on each $U_\alpha,$ where $\psi^\alpha_0,\cdots,\psi^\alpha_{n}$ are holomorphic functions on  $U_\alpha$ satisfying
$${\rm{codim}}_{\mathbb C}(\psi^\alpha_0=\cdots=\psi^\alpha_n=0)\geq2.$$
 Let $[w_0:\cdots:w_{n}]$  denote the homogeneous coordinate  of $\mathbb P^n(\mathbb C).$ Assuming that $w_0\circ \psi\not\equiv0.$
 Let $i:\mathbb C^n\hookrightarrow\mathbb P^n(\mathbb C)$ be an inclusion given by
 $(z_1,\cdots,z_n)\mapsto[1:z_1:\cdots:z_n].$
Clearly, $\omega_{FS}$ induces a (1,1)-form $i^*\omega_{FS}=dd^c\log(|\zeta_0|^2+|\zeta_1|^2+\cdots+|\zeta_n|^2)$ on $\mathbb C^n,$ where $\zeta_j:=w_j/w_0$ for $0\leq j\leq n.$
  The characteristic function of $\psi$ with respect to $i^*\omega_{FS}$ is  well defined by
\begin{eqnarray*}
 \hat{T}_\psi(r,\omega_{FS}) &=& \frac{1}{4}\int_{B_o(r)}g_r(o,x)\Delta_M\log\Big{(}\sum_{j=0}^n|\zeta_{j}\circ \psi(x)|^2\Big{)}dV(x).
\end{eqnarray*}
Clearly,
$$ \hat{T}_\psi(r,\omega_{FS})\leq\frac{1}{4}\int_{B_o(r)}g_r(o,x)\Delta_M\log\|\psi(x)\|^2dV(x)=T_\psi(r,\omega_{FS}).$$
The co-area formula leads to
$$  \hat{T}_\psi(r,\omega_{FS})=\frac{1}{4}\mathbb E_o\Big{[}\int_0^{\tau_r}\Delta_M\log\Big{(}\sum_{j=0}^n|\zeta_{j}\circ\psi(X_{t})|^2\Big{)}dt\Big{]}.
$$
 Note that the  pole divisor of $\zeta_{j}\circ \psi$ is pluripolar. By Dynkin formula
\begin{eqnarray*}
\hat{T}_\psi(r,\omega_{FS})&=&\frac{1}{2}\int_{S_o(r)}\log\Big{(}\sum_{j=0}^n|\zeta_{j}\circ\psi(x)|^2\Big{)}d\pi^r_o(x)
-\frac{1}{2}\log\Big{(}\sum_{j=0}^n|\zeta_j\circ\psi(o)|^2\Big{)}, \\
 \hat{T}_{\zeta_{j}\circ\psi}(r, \omega_{FS})
&=&\frac{1}{2}\int_{S_o(r)}\log\big{(}1+|\zeta_{j}\circ\psi(x)|^2\big{)}d\pi^r_o(x)-\frac{1}{2}\log\big{(}1+|\zeta_{j}\circ\psi(o)|^2\big{)}.
\end{eqnarray*}
\begin{theorem}\label{v000} We have
$$\max_{1\leq j\leq n}T(r,\zeta_{j}\circ\psi)+O(1)\leq \hat{T}_\psi(r,\omega_{FS})\leq\sum_{j=1}^nT(r,\zeta_{j}\circ\psi)+O(1).$$
\end{theorem}
\begin{proof}
On the one hand, 
\begin{eqnarray*}
  &&\hat{T}_\psi(r,\omega_{FS})  \\
  &\leq& \frac{1}{2}\sum_{j=1}^n\Big{(}\int_{S_o(r)}\log\big{(}1+|\zeta_{j}\circ\psi(x)|^2\big{)}d\pi^r_o(x)-\log\big{(}1+|\zeta_{j}\circ\psi(o)|^2\big{)}\Big{)}+O(1) \\
  &=& \sum_{j=1}^nT(r,\zeta_{j}\circ\psi)+O(1).
  \end{eqnarray*}
On the other hand, 
\begin{eqnarray*}
 T(r,\zeta_{j}\circ \psi)
 &=&\hat{T}_{\zeta_{j}\circ\psi}(r,\omega_{FS})+O(1)     \\
 &\leq& \frac{1}{4}\int_{B_o(r)}g_r(o,x)\Delta_M\log\Big{(}\sum_{j=0}^n|\zeta_{j}\circ\psi(x)|^2\Big{)}dV(x)+O(1) \\
  &=& \hat{T}_{\psi}(r,\omega_{FS})+O(1).
\end{eqnarray*}
We conclude the proof.
\end{proof}

\begin{cor}\label{seea} We have
$$\max_{1\leq j\leq n}T(r,\zeta_{j}\circ\psi)\leq T_\psi(r,\omega_{FS})+O(1).$$
\end{cor}
 Let  $V$ be a complex projective algebraic variety
and  $\mathbb{C}(V)$ be the field of rational functions defined on $V$ over $\mathbb C.$
 Let $V\hookrightarrow \mathbb P^N(\mathbb C)$ be a holomorphic embedding, and let $H_V$ be the restriction of hyperplane line bundle $H$ over $\mathbb P^N(\mathbb C)$ to $V.$
Denote by $[w_0:\cdots:w_{N}]$ the homogeneous coordinate system of $\mathbb P^N(\mathbb C)$ and assume that $w_0\neq0$ without loss of generality. Notice that the restriction $\{\zeta_j:=w_j/w_0\}$ to $V$ gives a transcendental base of $\mathbb{C}(V).$ Thereby, any $\phi\in\mathbb C(V)$ can be represented by a rational function in $\zeta_1,\cdots,\zeta_N$
$$ \phi=Q(\zeta_1,\cdots,\zeta_N).
$$
\begin{theorem}\label{a00a} Let $f:M\rightarrow V$ be an algebraically  non-degenerate meromorphic mapping. Then for  $\phi\in\mathbb C(V),$  there is  a constant $C>0$ depending on $\phi$ such that
$$ T(r,\phi\circ f)\leq CT_f(r,H_V)+O(1).$$
\end{theorem}
\begin{proof}
Assume that $w_0\circ f\not\equiv0$ without loss of generality.
Since $Q_j$ is rational,   there is  constant $C'>0$ such that
$T(r,\phi\circ f)\leq C'\sum_{j=1}^N T(r,\zeta_j\circ f)+O(1).$
      By Corollary \ref{seea},
 $ T(r,\zeta_j\circ f)
      \leq T_f(r,H_V)+O(1).$ This proves the theorem.
\end{proof}
\begin{cor}\label{49} Let $f:M\rightarrow V$ be an algebraically non-degenerate meromorphic mapping. Fix  a positive  $(1,1)$-form $\omega$ on $V.$
Then for any $\phi\in\mathbb C(V),$  there is a constant $C>0$ depending on $\phi$  such that
$$T(r,\phi\circ f)\leq CT_{f}(r,\omega)+O(1).$$
\end{cor}
\begin{proof} The compactness of $V$ and Theorem \ref{a00a}  implies  the assertion.
\end{proof}

\subsection{Estimate of $\mathbb E_o[\tau_r]$}~

We let $M$ be a simply-connected complete K\"ahler manifold of non-positive sectional curvature, and let $X_t$ be the Brownian motion in $M$ with generator $\Delta_M/2$ started at $o.$ Recall that $\dim_{\mathbb C}M=m,$ $\tau_r=\inf\{t>0:X_t\not\in B_o(r)\}.$
\begin{lemma}\label{yyy} We have
$$\mathbb E_o\big{[}\tau_r\big{]}\leq\frac{2r^2}{2m-1}.$$
\end{lemma}
\begin{proof}   The argument follows essentially from  Atsuji \cite{atsuji}, but here we provide a simpler proof though a rougher estimate.  We refer the reader to \cite{atsuji} for a better estimate that $\mathbb E_o[\tau_r]\leq r^2/2m.$ 
Let $X_t$ be the Brownian motion in $M$ started at $o\not=o_1,$ where $o_1\in B_o(r).$  Let  $r_1(x)$ be  the distance function  of $x$ from $o_1.$
Apply It\^o formula to $r_1(x)$
\begin{equation}\label{kiss}
  r_1(X_t)-r_1(X_0)=B_t-L_t+\frac{1}{2}\int_0^t\Delta_Mr_1(X_s)ds,
\end{equation}
here $B_t$ is the standard Brownian motion in $\mathbb R,$ and $L_t$ is a local time on cut locus of $o,$ an increasing process
which increases only at cut loci of $o.$ Since $M$ is simply connected and  non-positively  curved, then
$$\Delta_Mr_1(x)\geq\frac{2m-1}{r_1(x)}, \ \ L_t\equiv0.$$
By (\ref{kiss}), we arrive at
$$r_1(X_t)\geq B_t+\frac{2m-1}{2}\int_0^t\frac{ds}{r_1(X_s)}.$$
Let $t=\tau_r$ and take expectation on both sides of the above inequality, then it yields that 
$$\max_{x\in S_o(r)} r_1(x)\geq \frac{(2m-1)\mathbb E_o[\tau_r]}{2\max_{x\in S_o(r)} r_1(x)}.$$
Let $o'\rightarrow o,$ 
we are led to the conclusion.\end{proof}

\subsection{Second Main Theorem}~

Let $M$ be a complete K\"ahler manifold of non-positive sectional curvature. 
Consider  the (analytic) universal covering $$\pi:\tilde{M}\rightarrow M.$$ Via the pull-back by $\pi,$ $\tilde{M}$ can be equipped with the induced metric
 from the
metric of $M.$ So, under this metric, $\tilde{M}$ becomes a simply-connected complete K\"ahler manifold of non-positive sectional curvature.
Take a diffusion process $\tilde{X}_t$ in $\tilde{M}$  such that $X_t=\pi(\tilde{X}_t),$ here $X_t$ is the Brownian motion started at $o\in M,$ then
$\tilde{X}_t$ is a Brownian motion generated by $\Delta_{\tilde{M}}/2$ induced from the pull-back metric. Let $\tilde{X}_t$ start at $\tilde{o}\in\tilde{M}$  with $o=\pi(\tilde{o}).$ Then 
$$\mathbb E_o[\phi(X_t)]=\mathbb E_{\tilde{o}}\big{[}\phi\circ\pi(\tilde{X}_t)\big{]}$$
for $\phi\in \mathscr{C}_{\flat}(M).$ Set $$\tilde{\tau}_r=\inf\big{\{}t>0: \tilde{X}_t\not\in B_{\tilde{o}}(r)\big{\}},$$ where
$B_{\tilde{o}}(r)$ is a geodesic ball centered at $\tilde{o}$ with radius $r$ in $\tilde{M}.$
 If necessary, one can extend the filtration in probability space where $(X_t,\mathbb P_o)$ are defined so that $\tilde{\tau}_r$ is a stopping time with
 respect to a filtration where the stochastic calculus of $X_t$ works.
By the above arguments, we may assume $M$ is simply connected by lifting $f$ to the universal covering.

Let  $V$ be a complex projective algebraic manifold with complex dimension $n\leq m=\dim_{\mathbb C}M,$
 and let $L\rightarrow V$ be a holomorphic line bundle. Let a divisor $D\in |L|$ be of simple normal crossing type,  one can express
 $D=\sum_{j=1}^qD_j$ as the union of irreducible components.
 Equipping $L_{D_j}$ with Hermitian
metric which then induces a natural Hermitian metric $h$ on $L=\otimes_{j=1}^qL_{D_j}.$
Fixing a Hermitian metric form $\omega$ on $V$, which gives  a (smooth) volume form $\Omega:=\omega^n$ on $V.$
Pick $s_j\in H^0(V,L_{D_j})$
with $(s_j)=D_j$ and $\|s_j\|<1.$
On $V$, one defines a singular volume form
\begin{equation}\label{1phi}
  \Phi=\frac{\Omega}{\prod_{j=1}^q\|s_j\|^2}.
\end{equation}
Set
$$\xi\alpha^m=f^*\Phi\wedge\alpha^{m-n}.$$
Note that
$$\alpha^m=m!\det(g_{i\bar j})\bigwedge_{j=1}^m\frac{\sqrt{-1}}{2\pi}dz_j\wedge d\bar z_j.$$
A direct computation leads to
$$dd^c\log\xi\geq f^*c_1(L,h)-f^*{\rm{Ric}} \Omega+\mathscr{R}_M-{\rm{Supp}}f^*D$$
in the sense of currents, where $\mathscr{R}_M=-dd^c\log\det(g_{i\bar j}).$
This follows that
\begin{eqnarray}\label{5q}
&&\frac{1}{4}\int_{B_o(r)}g_r(o,x)\Delta_M\log\xi(x) dV(x) \\
&\geq& T_f(r,L)+T_f(r,K_V)+T(r,\mathscr{R}_M)-\overline{N}_f(r,D)+O(1). \nonumber
\end{eqnarray}

 We now prove Theorem \ref{second}:
\begin{proof}
By  Ru-Wong's arguments (see \cite{ru}, Page 231-233), the simple normal crossing type of $D$ implies that
there exists a finite  open covering $\{U_\lambda\}$ of $V$ together with rational functions
$w_{\lambda1},\cdots,w_{\lambda n}$ on $V$ for  $\lambda$ such that $w_{\lambda1},\cdots$ are holomorphic on $U_\lambda$ as well as
\begin{eqnarray*}
  dw_{\lambda1}\wedge\cdots\wedge dw_{\lambda n}(y)\neq0, & & \ ^\forall y\in U_{\lambda}, \\
  D\cap U_{\lambda}=\big{\{}w_{\lambda1}\cdots w_{\lambda h_\lambda}=0\big{\}}, && \ ^\exists h_{\lambda}\leq n.
\end{eqnarray*}
In addition, we  can require   $L_{D_j}|_{U_\lambda}\cong U_\lambda\times \mathbb C$ for 
$\lambda,j.$ On  $U_\lambda,$ we get
$$\Phi=\frac{e_\lambda}{|w_{\lambda1}|^2\cdots|w_{\lambda h_{\lambda}}|^2}
\bigwedge_{k=1}^n\frac{\sqrt{-1}}{2\pi}dw_{\lambda k}\wedge d\bar w_{\lambda k},$$
where $\Phi$ is given by (\ref{1phi}) and $e_\lambda$ is a smooth positive function.  
Let $\{\phi_\lambda\}$ be a partition of unity subordinate to $\{U_\lambda\},$ then $\phi_\lambda e_\lambda$ is bounded on $V.$ Set 
$$\Phi_\lambda=\frac{\phi_\lambda e_\lambda}{|w_{\lambda1}|^2\cdots|w_{\lambda h_{\lambda}}|^2}
\bigwedge_{k=1}^n\frac{\sqrt{-1}}{2\pi}dw_{\lambda k}\wedge d\bar w_{\lambda k}.$$
Put $f_{\lambda k}=w_{\lambda k}\circ f$, then on  $f^{-1}(U_\lambda)$ we obtain
 \begin{equation}\label{56q}
   f^*\Phi_\lambda=
   \frac{\phi_{\lambda}\circ f\cdot e_\lambda\circ f}{|f_{\lambda1}|^2\cdots|f_{\lambda h_{\lambda}}|^2}
   \bigwedge_{k=1}^n\frac{\sqrt{-1}}{2\pi}df_{\lambda k}\wedge d\bar f_{\lambda k}.
 \end{equation}
 Set
$$f^*\Phi\wedge\alpha^{m-n}=\xi\alpha^m, \ \ \ f^*\Phi_\lambda\wedge\alpha^{m-n}=\xi_\lambda\alpha^m$$
which arrives at  (\ref{5q}). Clearly, we have $\xi=\sum_\lambda \xi_\lambda.$ Again, set
\begin{equation}\label{gtou}
  f^*\omega\wedge\alpha^{m-1}=\varrho\alpha^m
\end{equation}
 which follows that
\begin{equation}\label{d444}
  \varrho=\frac{1}{2m}e_{f^*\omega}.
\end{equation}
For each $\lambda$ and any $x\in f^{-1}(U_{\lambda}),$
take a  local holomorphic coordinate system $z$ around $x.$
Since $\phi_\lambda\circ f\cdot e_\lambda\circ f$ is bounded, it is not very hard to see from (\ref{56q}) and (\ref{gtou}) that $\xi_\lambda$ is
 bounded from above by $P_\lambda,$
where $P_\lambda$ is a polynomial in
$$\varrho, \ \ g^{i\bar{j}}\frac{\partial f_{\lambda k}}{\partial z_i}\overline{\frac{\partial f_{\lambda k}}{\partial z_j}}\Big{/}|f_{\lambda k}|^2, \ \ 1\leq i, j\leq m, \ 1\leq k\leq n.$$
This yields that
\begin{equation}\label{gx}
  \log^+\xi_\lambda\leq O\Big{(}\log^+\varrho+\sum_k\log^+\frac{\|\nabla_M f_{\lambda k}\|}{|f_{\lambda k}|}\Big{)}+O(1).
\end{equation}
Thus, we conclude that
\begin{eqnarray}\label{fill}
   \log^+\xi
   &\leq& O\Big{(}\log^+\varrho+\sum_{k, \lambda}\log^+\frac{\|\nabla_M f_{\lambda k}\|}{|f_{\lambda k}|}\Big{)}+O(1) 
\end{eqnarray}
on $M.$ On the one hand,
\begin{eqnarray*}
\frac{1}{4}\int_{B_o(r)}g_r(o,x)\Delta_M\log\xi(x) dV(x)
&=&\frac{1}{2}\mathbb E_o\big{[}\log\xi(X_{\tau_r})\big{]}+O(1)
\end{eqnarray*}
due to co-area formula and Dynkin formula. Hence, by (\ref{5q}) we have
\begin{eqnarray}\label{pfirst}
 &&\frac{1}{2}\mathbb E_o\big{[}\log\xi(X_{\tau_r})\big{]} \\
 &\geq& T_f(r,L)+T_f(r,K_V)+T(r,\mathscr{R}_M)  -\overline{N}_f(r,D) +O(1). \nonumber
\end{eqnarray}
On the other hand,
since $f_{\lambda k}$ is the pull-back of rational function $w_{\lambda k}$ on  $V$ by $f$,   Corollary \ref{49} implies that
\begin{equation}\label{hbhb}
  T(r,f_{\lambda k})\leq O(T_f(r,\omega))+O(1).
\end{equation}
Using  (\ref{gx}) and (\ref{hbhb}) with Theorem \ref{log1}, 
\begin{eqnarray*}
   & & \frac{1}{2}\mathbb E_o\big{[}\log\xi(X_{\tau_r})\big{]} \\
   &\leq& O\Big{(}\sum_{k,\lambda}\mathbb E_o\left[\log^+\frac{\|\nabla_M f_{\lambda k}\|}{|f_{\lambda k}|}(X_{\tau_r})\right]\Big{)}+O\big{(}\mathbb E_o\left[\log^+\varrho(X_{\tau_r})\right]\big{)}+O(1) \\
   &\leq& O\Big{(}\sum_{k,\lambda}m\Big{(}r,\frac{\|\nabla_M f_{\lambda k}\|}{|f_{\lambda k}|}\Big{)}\Big{)}+O\big{(}\log^+\mathbb E_o\left[\varrho(X_{\tau_r})\right]\big{)}+O(1) \\
   &\leq& O\Big{(}\sum_{k,\lambda}\log T(r,f_{\lambda k})+\log C(o,r,\delta)\Big{)}
   +O\big{(}\log^+\mathbb E_o\big{[}\varrho(X_{\tau_r})\big{]}\big{)} \\
   &\leq& O\big{(}\log T_f(r,\omega)+\log C(o,r,\delta)\big{)}+O\big{(}\log^+\mathbb E_o\big{[}\varrho(X_{\tau_r})\big{]}\big{)}.
  \end{eqnarray*}
In the meanwhile, Lemma \ref{calculus} and (\ref{d444}) imply
\begin{eqnarray*}
\log^+\mathbb E_o\big{[}\varrho(X_{\tau_r})\big{]}
&\leq& (1+\delta)^2\log^+\mathbb E_o\left[\int_0^{\tau_r}\varrho(X_t)dt\right]+\log C(o,r,\delta) \\
   &=& \frac{(1+\delta)^2}{2m}\log^+\mathbb E_o\left[\int_0^{\tau_r}e_{f^*\omega}(X_t)dt\right]+\log C(o,r,\delta) \\
   &\leq& \frac{(1+\delta)^2}{m}\log T_f(r,\omega)+\log C(o,r,\delta).
\end{eqnarray*}
By this with (\ref{pfirst}), we prove the theorem.
\end{proof}
We proceed to prove Theorem \ref{nonpositive}.
\begin{lemma}\label{b0}
Let $\kappa$ be defined by $(\ref{kappa}).$ If $M$ is non-positively curved, then
$$T(r,\mathscr{R}_M)\geq 2m\kappa(r)r^2.$$
\end{lemma}
\begin{proof} Lemma \ref{s123}  implies that
$0\geq s_M\geq mR_M.$
By co-area formula
\begin{eqnarray*}
T(r,\mathscr{R}_M)&=&-{1\over4}\mathbb E_o\left[\int_0^{\tau_r}\Delta_M\log\det(g_{i\bar{j}}(X_t))dt\right] \\
&=&\mathbb E_o\left[\int_0^{\tau_r}s_{M}(X_t)dt\right] \geq m\mathbb E_o\left[\int_0^{\tau_r}R_M(X_t)dt\right] \\
&\geq& m(2m-1)\kappa(r)\mathbb E_o[\tau_r].
\end{eqnarray*}
 We have $\mathbb E_o[\tau_r]\leq 2r^2/(2m-1)$ by Lemma \ref{yyy}. The proof is completed.
\end{proof}
\begin{proof} With the estimate of  $C(o, r, \delta)$ given by (\ref{esti}) and estimate of  $T(r,\mathscr{R}_M)$ given by Lemma \ref{b0}, Theorem \ref{nonpositive} follows from Theorem \ref{second}.
\end{proof}
 If $M=\mathbb C^m,$ then  $\kappa\equiv0.$ Theorem \ref{nonpositive} implies that
\begin{cor}[Carlson-Griffiths, \cite{gri}; Noguchi, \cite{Nchi}]\label{sd}  Let a divisor $D\in|L|$ be of simple normal
crossing type.
Let $f:\mathbb C^m\rightarrow V$ be a differentiably non-degenerate meromorphic mapping. Then
  \begin{eqnarray*}
     T_f(r,L)+T_f(r,K_V)
     \leq \overline{N}_f(r,D)+O\big{(}\log T_f(r,\omega)+\delta\log r\big{)} \ \big{\|}.
  \end{eqnarray*}
\end{cor}

\section{Second Main Theorem for singular divisors}
We extend the Second Main  Theorem for divisors of simply normal crossing type  to  general divisors.
Given a hypersurface $D$  of a complex projective algebraic manifold $V.$ Let $S$ denote the set of the points of $D$ at which $D$ has a
non-normal-crossing singularity. By  Hironaka's resolution of singularities (see \cite{Hir}), there exists a proper modification
$$\tau:\tilde{V}\rightarrow V$$
such that $\tilde{V} \setminus \tilde{S}$ is biholomorphic to $V\setminus S$
 under $\tau,$ and $\tilde{D}$ is only of normal crossing singularities,
  where $\tilde{S}=\tau^{-1}(S)$ and $\tilde{D}=\tau^{-1}(D)$. Let
  $\hat{D}=\overline{\tilde{D}\setminus\tilde{S}}$ be the closure of $\tilde{D}\setminus\tilde{S},$ and   $\tilde{S}_j$ be the irreducible components of $\tilde{S}.$ Put
\begin{equation}\label{71}
  \tau^*D=\hat{D}+\sum p_j\tilde{S}_j=\tilde{D}+\sum(p_j-1)\tilde{S}_j, \  \ R_\tau=\sum q_j\tilde{S}_j,
\end{equation}
 where $R_\tau$ is ramification divisor of $\tau,$ and $p_j,q_j>0$ are integers. Again, set
\begin{equation}\label{90}
  S^*=\sum\varsigma_j\tilde{S}_j, \ \  \varsigma_j=\max\big{\{}p_j-q_j-1,0\big{\}}.
\end{equation}
We endow $L_{S^*}$ with a Hermitian metric $\|\cdot\|$ and take  a holomorphic section $\sigma$ of $L_{S^*}$ with ${\rm{Div}}\sigma=(\sigma)=S^*$ and $\|\sigma\|<1.$
Let
$$f:M\rightarrow V$$
be a meromorphic mapping from a complete K\"ahler manifold $M$ such that $f(M)\not\subset D.$ The \emph{proximity function} of $f$ with respect to the singularities of $D$ is defined by
$$m_f\big{(}r,{\rm{Sing}}(D)\big{)}=\int_{S_o(r)}\log\frac{1}{\|\sigma\circ\tau^{-1}\circ f(x)\|}d\pi^r_o(x).$$
Let $\tilde{f}:M\rightarrow\tilde{V}$ be the lift of $f$ given by $\tau\circ \tilde{f}=f.$ Then, we verify that
\begin{equation}\label{100}
  m_f\big{(}r,{\rm{Sing}}(D)\big{)}=m_{\tilde{f}}(r,S^*)=\sum\varsigma_jm_{\tilde{f}}(r,\tilde{S}_j).
\end{equation}

We now prove  Theorem \ref{second1}:
\begin{proof} We first suppose that $D$ is the union of smooth hypersurfaces, namely, no irreducible component of $\tilde{D}$ crosses itself. Let $E$ be the union of generic hyperplane sections of $V$ so that the set
$A=\tilde{D}\cup E$ has only normal-crossing singularities. By (\ref{71}) with
$K_{\tilde{V}}=\tau^*K_V\otimes L_{R_\tau},$
we have
\begin{equation}\label{73}
  K_{\tilde{V}}\otimes L_{\tilde{D}}=\tau^*K_V\otimes\tau^*L_D\otimes\bigotimes L_{\tilde{S}_j}^{\otimes(1-p_j+q_j)}.
\end{equation}
Applying Theorem \ref{nonpositive} to $\tilde{f}$ for divisor $A,$
 \begin{eqnarray*}
     & & T_{\tilde{f}}(r,L_A)+T_{\tilde{f}}(r,K_{\tilde{V}}) \\
     &\leq& \overline{N}_{\tilde{f}}(r,A)+O\big{(}\log T_{\tilde{f}}(r,\tau^*\omega)-r^2\kappa(r)+\delta\log r\big{)}.
  \end{eqnarray*}
 The First Main Theorem implies that
\begin{eqnarray*}
   T_{\tilde{f}}(r,L_A)&=& m_{\tilde{f}}(r,A)+N_{\tilde{f}}(r,A)+O(1) \\
    &=& m_{\tilde{f}}(r,\tilde{D})+m_{\tilde{f}}(r,E)+N_{\tilde{f}}(r,A)+O(1) \\
    &\geq& m_{\tilde{f}}(r,\tilde{D})+N_{\tilde{f}}(r,A)+O(1) \\
    &=& T_{\tilde{f}}(r,L_{\tilde{D}})-N_{\tilde{f}}(r,\tilde{D})+N_{\tilde{f}}(r,A)+O(1),
\end{eqnarray*}
which leads to
$$T_{\tilde{f}}(r,L_A)-\overline{N}_{\tilde{f}}(r,A)\geq  T_{\tilde{f}}(r,L_{\tilde{D}})
-\overline{N}_{\tilde{f}}(r,\tilde{D})+O(1).$$
Combining $T_{\tilde{f}}(r,\tau^*\omega)=T_{f}(r,\omega)$ and $\overline{N}_{\tilde{f}}(r,\tilde{D})=\overline{N}_{f}(r,D)$ with the above,
\begin{eqnarray}\label{75}
     & & T_{\tilde{f}}(r,L_{\tilde{D}})+T_{\tilde{f}}(r,K_{\tilde{V}})\\
     &\leq& \overline{N}_{\tilde{f}}(r,\tilde{D})+O\big{(}\log T_{f}(r,\omega)-r^2\kappa(r)+\delta\log r\big{)}. \nonumber
  \end{eqnarray}
 It yields from  (\ref{73}) that
\begin{eqnarray}\label{76}
   && T_{\tilde{f}}(r,L_{\tilde{D}})+T_{\tilde{f}}(r,K_{\tilde{V}}) \\
   &=& T_{\tilde{f}}(r,\tau^*L_D)+T_{\tilde{f}}(r,\tau^*K_{V})+\sum(1-p_j+q_j)T_{\tilde{f}}(r,L_{\tilde{S}_j}) \nonumber \\
    &=& T_{f}(r,L_D)+T_{f}(r,K_{V})+\sum(1-p_j+q_j)T_{\tilde{f}}(r,L_{\tilde{S}_j}). \nonumber
\end{eqnarray}
Since $N_{\tilde{f}}(r,\tilde{S})=0,$  it follows from (\ref{90}) and (\ref{100}) that
\begin{eqnarray}\label{77}
    \sum(1-p_j+q_j)T_{\tilde{f}}(r,L_{\tilde{S}_j})
   &=&\sum(1-p_j+q_j)m_{\tilde{f}}(r,\tilde{S}_j)+O(1)  \\
   &\leq& \sum\varsigma_jm_{\tilde{f}}(r,\tilde{S}_j)+O(1) \nonumber \\
   &=& m_{f}\big{(}r,{\rm{Sing}}(D)\big{)}+O(1). \nonumber
\end{eqnarray}
Combining (\ref{75})-(\ref{77}), we show the theorem.

To prove the general case, according to the above proved, one only needs to verify this claim for an arbitrary  hypersurface $D$ of normal-crossing type. Note  by the arguments
in [\cite{Shiff}, Page 175] that there is a proper  modification $\tau:\tilde{V}\rightarrow V$ such that $\tilde{D}=\tau^{-1}(D)$ is only the union of a collection of smooth hypersurfaces of
 normal crossings. Thus,
$m_f(r,{\rm{Sing}}(D))=0.$ By the special case of this theorem proved, the claim  holds for $D$ by using Theorem \ref{nonpositive}.
\end{proof}
\begin{cor}[Shiffman, \cite{Shiff}]\label{d123} Let $D\subset V$ be an ample hypersurface.
Let $f:\mathbb C^m\rightarrow V$ be a differentiably non-degenerate meromorphic mapping. Then
  \begin{eqnarray*}
     & & T_f(r,L_D)+T_f(r,K_V) \\
     &\leq& \overline{N}_f(r,D)+m_f\big{(}r,{\rm{Sing}}(D)\big{)} +O\big{(}\log T_f(r,L_D)+\delta\log r\big{)} \ \big{\|}.
  \end{eqnarray*}
\end{cor}
\begin{cor}[Defect relation]\label{d73} Assume the same conditions as in Theorem $\ref{second1}.$
 If $f$ satisfies the growth condition
$$ \liminf_{r\rightarrow\infty}\frac{r^{2}\kappa(r)}{T_f(r,\omega)}=0,$$
where $\kappa$ is defined by $(\ref{kappa}),$ then
$$\Theta_f(D)\underline{\left[\frac{c_1(L)}{\omega}\right]}\leq
\overline{\left[\frac{c_1(K^*_V)}{\omega}\right]}+\limsup_{r\rightarrow\infty}\frac{m_f\big{(}r,{\rm{Sing}}(D)\big{)}}{T_f(r,\omega)}.$$
\end{cor}
For further consideration of defect relations, we introduce some additional  notations. Let $A$   be a hypersurface of $V$ such that $A\supset S,$
where $S$ is a set of non-normal-crossing singularities of $D$ given before. We write
\begin{equation}\label{e34}
  \tau^*A=\hat{A}+\sum t_j\tilde{S}_j, \ \  \hat{A}=\overline{\tau^{-1}(A)\setminus\tilde{S}}.
\end{equation}
Set
\begin{equation}\label{78}
  \gamma_{A,D}=\max\frac{\varsigma_j}{t_j}
\end{equation}
where $\varsigma_j$ are given by (\ref{90}). Clearly, $0\leq \gamma_{A,D}<1.$
Note from (\ref{e34}) that
$$m_f(r,A)=m_{\tilde{f}}(r,\tau^*A)\geq\sum t_jm_{\tilde{f}}(r,\tilde{S}_j)+O(1).$$
By (\ref{100}), we see that
\begin{equation}\label{710}
   m_{f}\big{(}r,{\rm{Sing}}(D)\big{)}\leq \gamma_{A,D}\sum t_jm_{\tilde{f}}(r,\tilde{S}_j)\leq \gamma_{A,D}m_f(r,A)+O(1).
\end{equation}
\begin{theorem}\label{d555}
Let $L\rightarrow V$ be a holomorphic line bundle, and let $D_1,\cdots,D_q$ $\in |L|$ be hypersurfaces such that any two among  them have no common
components. Let $A\subset V$ be a hypersurface containing the non-normal-crossing singularities of $\sum_{j=1}^q D_j.$
Let $f:M\rightarrow V$ be a differentiably non-degenerate meromorphic mapping. If $f$ satisfies the growth condition
$$ \liminf_{r\rightarrow\infty}\frac{r^{2}\kappa(r)}{T_f(r,\omega)}=0,$$
where $\kappa$ is defined by $(\ref{kappa}),$ then
$$\sum_{j=1}^q\Theta_f(D_j)\underline{\left[\frac{c_1(L)}{\omega}\right]}\leq
\frac{1}{q}\overline{\left[\frac{c_1(K^*_V)}{\omega}\right]}+\frac{\gamma_{A,D}}{q}\overline{\left[\frac{c_1(L_A)}{\omega}\right]}.$$
\end{theorem}
\begin{proof} By (\ref{710}), we get
$$\sum_{j=1}^q\limsup_{r\rightarrow\infty}\frac{m_f\big{(}r,{\rm{Sing}}(D_j)\big{)}}{T_f(r,\omega)}\leq
 \gamma_{A,D}\overline{\left[\frac{c_1(L_A)}{\omega}\right]}.$$
 Note that $L_{D_1+\cdots+D_q}=L^{\otimes q}.$ By Theorem \ref{d73}, we show the theorem.
\end{proof}

\begin{cor}[Shiffman, \cite{Shiff}]\label{shide}
Let $L\rightarrow V$ be a positive line bundle, and let $D_1,\cdots,D_q\in |L|$ be hypersurfaces such that any two among them have no common
components. Let $A\subset V$ be a hypersurface containing the non-normal-crossing singularities of $\sum_{j=1}^q D_j.$
Let $f:\mathbb C^m\rightarrow V$ be a differentiably non-degenerate meromorphic mapping.
Then
$$\sum_{j=1}^q\Theta_f(D_j)\leq\frac{1}{q}
\overline{\left[\frac{c_1(K^*_V)}{c_1(L)}\right]}+\frac{\gamma_{A,D}}{q}\overline{\left[\frac{c_1(L_A)}{c_1(L)}\right]}.$$
\end{cor}
\begin{proof} Replace $\omega$ by $c_1(L,h)$ in Theorem \ref{d555}.
\end{proof}

\begin{cor}\label{f75}
Let $L\rightarrow V$ be a positive  line bundle, and let $D\in |L|$ be a hypersurface. If there is a hypersurface  $A\subset V$ containing the non-normal-crossing singularities of $D$ such that
$$\overline{\left[\frac{c_1(K^*_V)}{c_1(L)}\right]}+\gamma_{A,D}\overline{\left[\frac{c_1(L_A)}{c_1(L)}\right]}<1.$$
Then every meromorphic mapping $f:M\rightarrow V\setminus D$
satisfying
$$ \liminf_{r\rightarrow\infty}\frac{r^{2}\kappa(r)}{T_f(r,L)}=0$$
is differentiably degenerate, where $\kappa$ is defined by $(\ref{kappa}).$ 
\end{cor}

\begin{cor}  Let $D\subset\mathbb P^n(\mathbb C)$ be a hypersurface
of degree $d_D.$ If there is  a hypersurface  $A\subset\mathbb P^n(\mathbb C)$ of degree $d_A$
 containing the non-normal-crossing singularities of $D$
such that
$d_A\gamma_{A,D}+n+1<d_D.$
Then every meromorphic mapping $f:M\rightarrow \mathbb P^n(\mathbb C)\setminus D$
satisfying
$$ \liminf_{r\rightarrow\infty}\frac{r^{2}\kappa(r)}{T_f(r,L_D)}=0$$
 is differentiably degenerate, where $\kappa$ is defined by $(\ref{kappa}).$ 
\end{cor}
\begin{proof} The conditions imply that
$$\overline{\left[c_1(K^*_{\mathbb P^n(\mathbb C)})/c_1([D])\right]}+\gamma_{A,D}\overline{\left[c_1([A])/c_1([D])\right]}= \frac{n+1}{d_D}+\gamma_{A,D}\frac{d_A}{d_D}<1.$$
By Corollary \ref{f75}, we see that the corollary holds.
\end{proof}

\vskip\baselineskip

\label{lastpage-01}
\end{document}